\documentclass{amsart}
\usepackage{amsfonts,amssymb}
\usepackage{enumerate,graphicx}
\usepackage{subcaption}
\usepackage[square,numbers]{natbib}
\usepackage{mathrsfs,bbm}
\usepackage{tikz,tikzscale,tikz-cd}
\usepackage{multirow,xparse,fp}
\usepackage{environ}
\usepackage{amstext}
\usepackage{hyperref}
\usepackage{diagbox}

\newtheorem{theorem}{Theorem}[section]

\newtheorem{proposition}[theorem]{Proposition}
\newtheorem{corollary}[theorem]{Corollary}
\newtheorem{lemma}[theorem]{Lemma}
\theoremstyle{definition}

\newtheorem{definition}[theorem]{Definition}
\newtheorem{example}[theorem]{Example}

\newtheorem{remark}[theorem]{Remark}
\newtheorem{question}[theorem]{Question}

\newcommand\alphabet{\mathcal{A}}
\newcommand\mcX{\mathsf{X}}
\newcommand\mcXp{\mathsf{X}_{\Omega}^{(\base)}}
\newcommand\mcXpp{\mathsf{X}_{\Omega}^{(2)}}
\newcommand\base{\mathit{l}}
\newcommand\baseq{\mathit{L}}
\newcommand\Pint{\mathbb{P}}
\newcommand\Zint{\mathbb{Z}}

\begin{document}

\title{Topologically Mixing Properties of Multiplicative Integer System}

\author[Jung-Chao Ban]{Jung-Chao Ban}
\address[Jung-Chao Ban]{Department of Mathematical Sciences, National Chengchi University, Taipei 11605, Taiwan, ROC.}
\address{Math. Division, National Center for Theoretical Science, National Taiwan University, Taipei 10617, Taiwan. ROC.}
\email{jcban@nccu.edu.tw}

\author[Chih-Hung Chang]{Chih-Hung Chang}
\address[Chih-Hung Chang]{Department of Applied Mathematics, National University of Kaohsiung, Kaohsiung 81148, Taiwan, ROC.}
\email{chchang@nuk.edu.tw}

\author[Wen-Guei Hu]{Wen-Guei Hu}
\address[Wen-Guei Hu]{College of Mathematics, Sichuan University, Chengdu, 610064, China}
\email{wghu@scu.edu.cn}

\author[Guan-Yu Lai]{Guan-Yu Lai}
\author[Yu-Liang Wu]{Yu-Liang Wu}
\address[Guan-Yu Lai and Yu-Liang Wu]{Department of Applied Mathematics, National Chiao Tung University, Hsinchu 30010, Taiwan, ROC.}
\email{guanyu.am04g@g2.nctu.edu.tw; s92077.am08g@nctu.edu.tw}

\thanks{Ban and Chang are partially supported by the Ministry of Science and Technology, ROC (Contract No MOST 107-2115-M-259 -001 -MY2 and 107-2115-M-390 -002 -MY2). Hu is partially supported by the National Natural Science Foundation of China (Grant No.11601355).}

\date{October 30, 2019}

\baselineskip=1.2\baselineskip

\begin{abstract}
Motivated from the study of multiple ergodic average, the investigation of multiplicative shift spaces has drawn much of interest among researchers. This paper focuses on the relation of topologically mixing properties between multiplicative shift spaces and traditional shift spaces. Suppose that $\mathsf{X}_{\Omega}^{(l)}$ is the multiplicative subshift derived from the shift space $\Omega$ with given $l > 1$. We show that $\mathsf{X}_{\Omega}^{(l)}$ is (topologically) transitive/mixing if and only if $\Omega$ is extensible/mixing. After introducing $l$-directional mixing property, we derive the equivalence between $l$-directional mixing property of $\mathsf{X}_{\Omega}^{(l)}$ and weakly mixing property of $\Omega$.
\end{abstract}
\maketitle

\section{Introduction}

Let $\mathcal{A} = \{0,1,\ldots ,m-1\}$ be a finite alphabet and let $\Omega \subseteq \mathcal{A}^{\mathbb{N}}$ be a shift space with the shift map $\sigma: \Omega \to \Omega$ defined by $(\sigma x)_{i} = x_{i+1}$ for $i \in \mathbb{N}$. Suppose $1 < l$ is a natural number. Kenyon et al.~\cite{KPS-ETDS2012} defined the \emph{multiplicative subshift} $\mcXp \subseteq \mathcal{A}^{\mathbb{N}}$ as
\begin{equation}
\mcXp=\{x=(x_{k})_{k=1}^{\infty }\in \mathcal{A}^{\mathbb{N}}:(x_{i l^{n-1}})_{n \in \mathbb{N}}\in \Omega \text{ for all } i\}.  \label{4}
\end{equation}%
The name ``\emph{multiplicative subshift}'' follows from the fact $\mcXp$ is multiplicatively invariant in the sense $\Pi_{q} \mcXp \subseteq \mcXp$ for all $q \in \mathbb{N}$, where $\Pi_{q}x = (x_{q k})_{k=1}^{\infty}$. The study of \eqref{4} takes its origin from the multifractal analysis of the $0$-level set $E_\Phi (0)$ considered by Fan et al.~\cite{FLM-MM2012}, where
\begin{equation}
E_{\Phi }(\theta)=\{x\in \{0, 1\}^{\mathbb{N}}:\lim_{n\rightarrow \infty } \frac{1}{n}\sum_{k=1}^{n}x_{k}x_{2k}=\theta\}. \label{2}
\end{equation}
The study of \eqref{2} is a special case of the multiple ergodic averages $A_{n}\Phi (x):=\frac{1}{n}\sum_{k=0}^{n-1}\Phi (T_{1}^{k}x,\ldots,T_{d}^{k}x)$ with $d=2$, $T_{i}= \sigma^i$, and $\Phi(x,y)=x_1 y_1$. In the same paper, they also studied the subset
\begin{equation}
\mcX_{2}=\{x\in \{0, 1\}^{\mathbb{N}}:x_{i}x_{2i}=0\text{ for all } i\}  \label{5}
\end{equation}%
of \eqref{2}. It is easily seen that $\mcX_{2}$ is a special type of $\mcXp$ in
which $\Omega$ is the \emph{golden mean shift} (it is called ``\emph{multiplicative golden mean shift}'' in \cite{KPS-ETDS2012}). Moreover, the box dimensions of $\mcX_{2}$ is $\dim _{B}\mcX_{2}=\frac{1}{2\log 2} \sum_{n=1}^{\infty }\frac{\log F_{n}}{2^{n}}$, where $F_{n}$ is the Fibonacci sequence: $F_{0}=1$, $F_{1}=2$ and $F_{n+2}=F_{n+1}+F_{n}$ for $n\geq 0$. Later, Kenyon et al.~\cite{KPS-ETDS2012} obtained the general formula of Hausdorff and box dimensions of $\mcX_{\Omega }$. It is known that the subshift corresponding to the closed invariant subsets of $[0,1]$ under the map $x\mapsto mx \pmod{1}$ has the property that the Hausdorff and box dimensions are coincident \cite{furstenberg1967disjointness}. Even though the Hausdorff and box dimension of $\mcX_{\Omega}$ are not coincident generally (the Hausdorff dimension is less than or equal to the box dimension), the characterization of the equality is addressed therein. Furthermore, Peres and Solomyak \cite{peres2011dimension} gave the full dimension spectrum of $\dim_{H}E_{\Phi }(\theta )$, and mentioned that $\dim _{H}\mcX_{2}=\dim_{H}E_{\Phi }(0)$. Beyond $\mcX_{2}$, more dimension results can be found. Peres et al.~\cite{peres2014dimensions} considered the multiplicative subshifts $\mcX_{\Omega }^{(S)}$, for which $S$ is the semigroup generated by primes $p_{1},\ldots ,p_{k}$. Namely, 
\begin{equation*}
\mcX_{\Omega }^{(S)}:=\{x=(x_{k})_{k=1}^{\infty }\in \mathcal{A}^{\mathbb{N}}:x|_{iS}\in \Omega \text{ for all } i \text{ such that }(i,S)=1\}\text{.}
\end{equation*}
A typical example of $\mcX_{\Omega }^{(S)}$ is 
\begin{equation*}
\mcX_{2,3}=\{x\in \{0, 1\}^{\mathbb{N}}:x_{k}x_{2k}x_{3k}=0\text{ for all } k\}\text{,}
\end{equation*}
where $S$ is the semigroup generated by $2$ and $3$. The authors in \cite{peres2014dimensions} extended \cite{KPS-ETDS2012} to $\mcX_{\Omega }^{(S)}$ and obtained the Hausdorff and box dimensions of $\mcX_{\Omega }^{(S)}$. Ban et al.~\cite{ban2017pattern} approximated the box dimension $\dim_{B}\left( \mcX_{\Omega}^{(S)}\cap \Omega \right)$ for the case where $\Omega$ is a subshift of finite type. Fan et al.~\cite{FSW-AM2016} gave a complete solution to the problem of multifractal analysis of the limit of the multiple ergodic averages $\frac{1}{n} \sum_{i=1}^n \phi(x_i, x_{il}, \ldots, x_{il^{k-1}})$ for $k, l \geq 2$. We refer to \cite{fan2014some} for a nice state-of-the-art survey of the the multiple ergodic averages.

\begin{figure}[t]
        \begin{tikzpicture}[scale=0.5, transform shape]
        \draw  (1,-0.825) ellipse (7.3 and 3.675) ;
        \node[] at (1,2.1) {\Large $\Omega$: extensible / $\mcXp$: transitive};
        
        \draw  (1,-1.3) ellipse (6.5 and 2.8);
        \node[] at (1,0.75) {\Large $\Omega$: transitive};
        
        \draw  (1,-1.7) ellipse (6 and 1.9);
        \node[] at (1,-0.65) {\Large $\Omega$: weakly mixing / $\mcXp$: $l$-directional mixing };
        
        \draw  (1,-2.25) ellipse (4.5 and 0.8);
        \node[] at (1,-2.25) {\Large $\Omega$: mixing / $\mcXp$: mixing};
        \end{tikzpicture}
         \caption{Summary of the result}
         \label{fig:summary}
\end{figure}
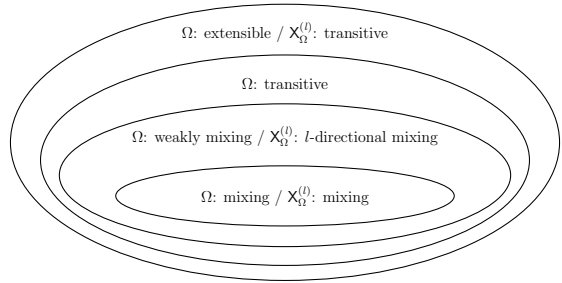

Besides the dimensional aspect of the multiplicative subshifts, the topological behaviors of $\mcXp$ or $\mcX_{\Omega }^{(S)}$ are also fascinating. This paper aims to connect the mixing properties of $\Omega $ and $\mcXp$. To be more specific, assuming $\Omega $ possesses some topological property \textbf{P}, can we say something about the mixing properties of $\mcXp$? Are the properties  $\mcXp$ equipped with stronger or weaker than property \textbf{P}? In other words, the goal of this paper is trying bridge the topological behaviors of the two spaces: one is additively invariant and the other is multiplicatively invariant. Theorems \ref{thm:extensibility_equivalence}, \ref{thm:weak_mixing_equivalence}, and \ref{thm:mixing_equivalence} are the main results of this paper (cf.~Figure \ref{fig:summary}).

\begin{theorem} \label{thm:extensibility_equivalence}
        The following are equivalent.
        
        \begin{enumerate}
        \item $\Omega$ is extensible.
        \item $\mcXp$ is transitive.
        \item For $u,v\in B(\mcXp)$ there exists $\alpha \in \mathbb{N} \setminus \base \mathbb{N}$ such that for any $k\in \mathbb{N}_{0}$ there exists $x\in \mcXp$ such that $x|_{s(u)}=u$ and $(\Pi_{\vert u \vert \alpha \base^k s(v)} x)
        |_{s(v)}=v$.
        \item $\mcXp$ is $\baseq$-directional mixing for some $\baseq$ which has a  prime factor $p \in \Pint$ satisfying $p \nmid \base$.
        \item $\mcXp$ is $\baseq$-directional mixing for every $\baseq$ which has a  prime factor $p \in \Pint$ satisfying $p \nmid \base$.
        \end{enumerate}
\end{theorem}

\begin{theorem}\label{thm:weak_mixing_equivalence}
        The following are equivalent. 

        \begin{enumerate}
        \item $\Omega $ is weakly mixing.
        \item $\mcXp$ is $\base$-directional mixing.
        \item $\mcXp$ is $\base^n$-directional mixing for every $n \in \mathbb{N}$.
        \item $\mcXp$ is $\base^n$-directional mixing for some $n \in \mathbb{N}$.
        \end{enumerate}
        Furthermore, if $\base$ is a prime number, the following is equivalent as above.
        \begin{enumerate}
        \setcounter{enumi}{4}
        \item $\mcXp$ is $\baseq$-directional mixing for every $\baseq > 1$.
        \end{enumerate}
\end{theorem}

\begin{theorem}\label{thm:mixing_equivalence}
        The following are equivalent.
        \begin{enumerate}
        \item  $\Omega$ is mixing.
        \item $\mcXp$ is mixing.
        \item For $u,v \in B(\mcXp)$, there exists $N \in \mathbb{N}$ such that for $k \geq N$ and $\alpha \in \mathbb{N} \setminus \base \mathbb{N}$ there exists $x \in \mcXp$ such that $x|_{s(u)}=u$ and $( \Pi_{\vert u \vert \alpha \base^{k}} x)|_{s(v)}=v$.
        \end{enumerate}
\end{theorem}

        \begin{remark} \label{rmk:omega_transitive}
                The following example distinguishes Theorems \ref{thm:extensibility_equivalence} and \ref{thm:weak_mixing_equivalence}. Let $\Omega = \mathsf{X}_\mathcal{F}$ be defined by forbidden set $\mathcal{F}=\{01\}$. Then, $\Omega$ is extensible yet not transitive, and $\mcXp$ satisfies Theorem \ref{thm:extensibility_equivalence}. However, $\mcXp$ is not $\base$-directional mixing. It can be verified by considering $u'=0^{\base},v'=1 \in B(\mcXp)$ and $\alpha = 1$. Under the circumstances, no $y \in \mcXp$ should accept $v'$ since $1$ at the the position where $v'$ is located is to the right of $0$.
\end{remark}

We organize the material of this paper as follows. Section 2 elucidates the definitions and propositions that are used in this investigation and gives some examples to illustrate the idea of the proof of the main theorems. Section 3 is devoted to the proofs of Theorems \ref{thm:extensibility_equivalence}, \ref{thm:weak_mixing_equivalence}, and \ref{thm:mixing_equivalence}. Discussions and some open problems are carried out in Section 4.

\section{Definitions and Examples}

Let $\alphabet$ be a finite alphabet and let $\Omega \subset \alphabet^{\mathbb{N}}$ be a shift space with the shift map $\sigma: \alphabet^{\mathbb{N}} \to \alphabet^{\mathbb{N}}$. Denote the set of all admissible words of length $n$ by $B_{n}(\Omega )$ and set $B(\Omega )=\cup _{n\geq 1} B_{n}(\Omega )$. For every $u \in B(\Omega )$ and $x \in \Omega$, let $\vert u \vert$ be the length of $u$ and $x_{[i,j]}=(x_{i},\ldots ,x_{j})$ be the \emph{projection} of $x$ on $[i, j] := \{n \in \mathbb{N}: i \leq n \leq j\}$. 

\begin{definition} \label{def:mixing-transitive-shift-space}
Let $(\Omega, \sigma)$ be a symbolic dynamical system. We say that $(\Omega, \sigma)$ is
\begin{enumerate}[(1)]
\item \emph{extensible} if for all $u\in B(\Omega )$ and for all $m\in \mathbb{N}$, there exists $x\in \Omega $ such that $x_{[m+1,m+\left\vert u\right\vert ]}=u$;
\item \emph{transitive} if for all $u, v \in B(\Omega)$ there exists $m\in \mathbb{N}$ and $x\in \Omega $ such that $x_{[1,\left\vert u\right\vert ]}=u$ and $x_{[|u|+m+1,|u|+m+\vert v \vert]}=v$;
\item \emph{totally transitive} if $\sigma ^{n}$ is transitive for all $n\geq 1$;
\item   \emph{weakly mixing} if for all $u_{1},u_{2},v_{1},v_{2}\in B(\Omega)$ there exist $m \in \mathbb{N}$ and configurations $x^{(1)}, x^{(2)} \in \Omega$ such that $x^{(i)}|_{[1, |u_i|]}=u_i$ and $x^{(i)}|_{[|u_i|+m+1, |u_i|+m+|v_i|]}=v_i$ for $i = 1, 2$;
\item \emph{mixing} if for all $u, v\in B(\Omega)$ there exists $N\in \mathbb{N}$ such that for $m\geq N$ there exists $x\in \Omega $ such that $x_{[1,\left\vert u\right\vert ]}=u$ and $x_{[|u|+m+1, |u|+m+\left\vert v\right\vert ]}=v$.
\end{enumerate}
\end{definition}

It follows from the definitions that 
\begin{equation*}
        \text{mixing}\Rightarrow \text{weakly mixing}\Rightarrow \text{totally transitive}%
        \Rightarrow \text{transitive}\Rightarrow \text{extensible.}
\end{equation*}

Roughly speaking, a weakly mixing system transits any pair of open sets to another pair simultaneously. Furstenberg \cite{furstenberg1967disjointness} demonstrated the property holds for finitely many sets.

        \begin{proposition}[See \cite{furstenberg1967disjointness}] \label{prop:weak_mixing}
                Suppose $\Omega$ is a shift space. Then following statements are equivalent.
                \begin{enumerate}[(a)]
                        \item $\Omega$ is weakly mixing.
                        \item For $\{u_1, u_2, \ldots, u_M\}, \{v_1, v_2, \ldots, v_M\} \subset B(\Omega)$ there exist  $m \in \mathbb{N}$ and configurations $\{x^{(i)}\}_{i=1}^M$ satisfying $x^{(i)}|_{s(u_i)}=u_i$ and $(\sigma^m x^{(i)})|_{s(v_i)}=v_i$ for $1 \leq i \leq M$.
                \end{enumerate}
        \end{proposition}

For each shift space $\Omega$ and natural number $l \geq 2$, the multiplicative shift space $\mathsf{X}_{\Omega}^{(l)}$ is defined as
$$
\mcXp=\{x=(x_{k})_{k=1}^{\infty }\in \mathcal{A}^{\mathbb{N}}:(x_{i l^{n-1}})_{n \in \mathbb{N}}\in \Omega \text{ for all }i\}.
$$
We say that $u \in \mathcal{A}^S$ is a \emph{pattern} in $\mathsf{X}_{\Omega}^{(l)}$ if there exist a finite set $S \subset \mathbb{N}$ and $x \in \mathsf{X}_{\Omega}^{(l)}$ such that $x|_S = u$, i.e., $x_i = u_i$ for $i \in S$. In this case, $S$ is called the \emph{support} of $u$ and is denoted by $s(u)$. The \emph{multiplicative map} $\Pi: \mathbb{N} \times \mathsf{X}_{\Omega}^{(l)}$ is defined as
$$
(\Pi_q x)_i := \Pi(q, x)_i = x_{qi} \quad \text{for} \quad i \in \mathbb{N}, x \in \mathsf{X}_{\Omega}^{(l)}.
$$
It is obvious that $\mathsf{X}_{\Omega}^{(l)}$ is invariant under the multiplicative map and thus $(\mathsf{X}_{\Omega}^{(l)}, \{\Pi_q\}_{q \in \mathbb{N}})$ is a dynamical system. Therefore, Definition \ref{def:mixing-transitive-shift-space} can be extended to $\mathsf{X}_{\Omega}^{(l)}$ in a similar vein.

\begin{definition}
Let $(\mathsf{X}_{\Omega}^{(l)}, \{\Pi_q\}_{q \in \mathbb{N}})$ be a multiplicative shift. We say that $(\mathsf{X}_{\Omega}^{(l)}, \{\Pi_q\}_{q \in \mathbb{N}})$ is
\begin{enumerate}[(1)]
\item \emph{transitive} if for $u, v \in B(\mathsf{X}_{\Omega}^{(l)})$ there exist $m \in \mathbb{N}$ and $x \in \mathsf{X}_{\Omega}^{(l)}$ such that $x|_{[1, |u|]}=u$ and $\left(\Pi_{|u| m} x\right)|_{[1, |v|]}=v$;
\item \emph{mixing} if for $u, v\in B(\mathsf{X}_{\Omega}^{(l)})$ there exists $N\in \mathbb{N}$ such that for $m\geq N$ there exists $x\in \mathsf{X}_{\Omega}^{(l)}$ such that $x_{[1,\left\vert u\right\vert ]}=u$ and $\left(\Pi_{|u| m} x\right)|_{[1, |v|]}=v$.
\end{enumerate}
\end{definition}

The idea of these definitions is to connect the patterns $u,v$ in $\mcXp$ under the action of multiplicative semigroup of positive integers. Observe that, for each given integer $l \geq 2$, every natural number $n$ has a unique decomposition $n = \alpha l^k$, where $\alpha \in \mathbb{N} \setminus l \mathbb{N}$ and $k \geq 0$. The following proposition comes from straightforward examination, and thus the proof is omitted.

\begin{proposition}\label{prop:mixing-transitive-multiplicative-shift}
Consider the multiplicative shift $(\mathsf{X}_{\Omega}^{(l)}, \{\Pi_q\}_{q \in \mathbb{N}})$. Then $\mathsf{X}_{\Omega}^{(l)}$ is 
\begin{enumerate}[(1)]
\item \emph{transitive} if and only if for $u,v\in B(\mcXp)$ there exists $(\alpha ,k)\in (\mathbb{N} \setminus \base \mathbb{N}) \times \mathbb{N}_{0}$ and $x \in \mcXp$ such that $x|_{s(u)}=u$ and $\left( \Pi_{\left\vert u\right\vert \alpha l^{k}}x\right) |_{s(v)}=v$;
\item \emph{mixing} if and only if for all $u,v\in B(\mcXp)$ there exists $N\in \mathbb{N}$ such that if $(\alpha ,k)\in (\mathbb{N} \setminus \base \mathbb{N}) \times \mathbb{N}_{0}$ with $\alpha l^k \geq N$ there exists $x\in \mcXp$ such that $x|_{s(u)}=u$ and $\left( \Pi_{\left\vert u\right\vert \alpha l^{k}}x\right) |_{s(v)}=v$.
\end{enumerate}
\end{proposition}
%

Without causing ambiguity, we define the multiplicative map $\Pi$ on the collection of patterns $B(\mcXp)$ as follows. For each $q \in \mathbb{N}$, the multiplicative map $\Pi_q: \cup_{S \subset \mathbb{N}} \alphabet^{S} \rightarrow \cup_{S \subset \mathbb{N}} \alphabet^{S}$ is defined as $(\Pi_q u)_{i} = u_{q i}$ for $i \in s(u) \cap [1,\frac{|u|}{q}]$. In other words,  Suppose $S$ is a subset of $\mathbb{N}$, we define $((\Pi_q u)|_{S})_i = u_{q j_i}$, where $j_i = \min \{k: k \in S \cap (j_{i-1}, \frac{|u|}{q}], k q \in s(u)\}$ and $j_0 = 0$. In other words, $\Pi_q u$ is an arithmetic subword of $u$ with tolerance $q$.

For a multiplicative shift $\mcXp$ and $m, n \in \mathbb{N}$, we define $m \simeq n$ if and only if $\dfrac{m}{n} = l^i$ for some $i \in \mathbb{Z}$. It follows immediately that $\simeq$ is an equivalence relation. With this we define $\Lambda_{[i]} = \{i l^k: k \in \mathbb{Z}\} \cap \mathbb{N}$ for each $i \in \mathbb{N}$, where $[i]$ denotes the equivalence class of $i$ with respect to $\simeq$. The following proposition comes straightforwardly from the fundamental theorem of arithmetic.

\begin{proposition} \label{prop:space_partition}
Let $\Lambda_{[i]}$ be defined as above. Then
        \begin{enumerate}
                \item $\{\Lambda_{[i]}\}_{i \in \mathbb{N}}$ is a partition of $\mathbb{N}$;
                \item for each $\alpha \in \mathbb{N} \setminus \base \mathbb{N}$, $\alpha \Lambda_{[i]} \subset \Lambda_{[j]}$ for some $[j] \neq [i]$. In particular, $\alpha \Lambda_{[i]} = \Lambda_{[\alpha i]}$ if $\base \in \Pint$, where $\Pint$ denotes the set consisting of all prime numbers.
        \end{enumerate}
\end{proposition}

For the sake of simplicity, we refer to $\Lambda_{[i]}$ as $\Lambda_i$ with an extra requirement that $i$ is the smallest element of $[i]$ for the rest of this paper unless otherwise stated. For every $u \in B(\mcXp)$, we define $u|_{\Lambda_i} \in B(\Omega)$ and $\xi(u)$ as $u|_{\Lambda_i} := (u_{i \base^{j-1}})_{j}$ for $1 \leq j \leq \lfloor \log_l \frac{|u|}{i} \rfloor$ and $\xi(u):=\max\{n:n \le \vert u \vert, \base \nmid n \}$, respectively. In addition, for $1 < q, N \in \mathbb{N}$, let $A_q:=\mathbb{N} \setminus q \mathbb{N}$ denote the set of positive integers which are not divisible by $q$ and $A_{q,N}:=A_q \cap [1,N]$.

\begin{figure}[t]
\includegraphics[]{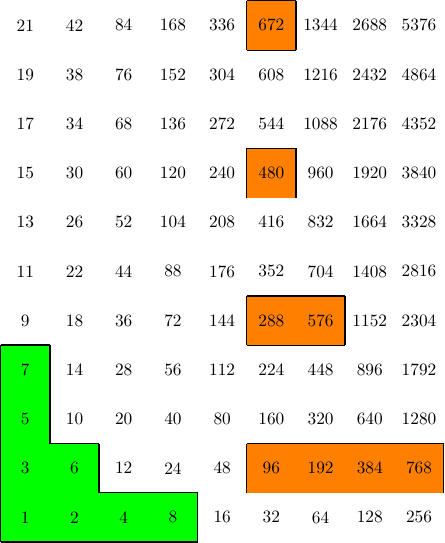}
        \caption{The multiplicative map $\Pi$ breaks the topological structure of the support of pattern in $\mathsf{X}_{\Omega}^{(l)}$. Suppose $l = 2$ and $u \in B_8(\mcXp)$. The support of $u$ is the green region and (part of) support of $(\Pi_{96} x)|_{s(u)}$ is colored in orange. Notably, $s(u)$ is connected while $96 s(u)$ is broken.}
        \label{fig:def_illustration}
\end{figure}

\begin{example}
        Suppose $\alphabet = \{0, 1\}$, $\base = 2$, and $\Omega$ is the one-sided golden mean shift. Then
        \begin{equation*}
                \mcX_2 := \mcXpp = \{x = (x_i)_{i \in \mathbb{N}}: x_i x_{2i} = 0 \text{ for all } i \in \mathbb{N}\}.
        \end{equation*}
        Suppose $u \in B_8 (\mcXp)$ and $x \in \mcXp$ with $(\Pi_{96} x)|_{s(u)}=u$. Then $\xi(u) = 7$, which means that $s(u) \cap \Lambda_7 \neq \varnothing$, see Figure \ref{fig:def_illustration}. The supports of $u$ and $x|_{96 s(u)} = (\Pi_{96} x)|_{s(u)}$ are colored in green and orange, respectively. Observe that the multiplicative map $\Pi$ breaks the topological structure of the support of pattern. In general, for any $u, v \in B(\mcXp)$ and $x \in \mcXp$ such that $(\Pi_{\vert u \vert \alpha \base^k} x)|_{s(v)}=v$ and $x|_{s(u)}=u$, the supports of $u$ and $x|_{\vert u \vert \alpha \base^k s(v)}$ are non-overlapping if $\alpha \base^k \ne 1$. 
\end{example}

One of the main difference between traditional shift spaces and multiplicative shift spaces is that the multiplicative map messes up the topological structure of the underlying space, which makes the investigation of dynamical phenomena of multiplicative shift spaces much more complicated and diversifies mixing properties in multiplicative subshifts. The following definition introduces a mixing property called \emph{directional mixing} that is related to the weakly mixing property in traditional shift spaces.

\begin{definition} \label{def:directional_mixing}
        Suppose $\mcXp$ is a multiplicative shift space. We say that $\mcXp$ is \emph{$q$-directional mixing} for some $q > 1$ if for $u, v \in B(\mcXp)$ there exists $k \in \mathbb{N}_0$ such that for any $\alpha \in A_q$ there exists $x \in \mcXp$ satisfying $x|_{s(u)} = u$ and $(\Pi_{\vert u \vert \alpha q^k} x)|_{s(v)} = v$.
\end{definition}

Theorem \ref{thm:extensibility_equivalence} reveals that $\Omega$ is extensible if and only if $\mcXp$ is transitive. Example \ref{ex:extensibility} yields an observation how Theorem \ref{thm:extensibility_equivalence} holds.

\begin{example} \label{ex:extensibility}
        Let $\Omega = \mcX_\mathcal{F} \subset \Sigma_{2}^\mathbb{N}$ be defined by forbidden set $\mathcal{F}=\{01,11\}$. Apparently, $\Omega$ is extensible. We give the equivalence between the extensibility of $\Omega$ and the transitivity of $\mcXp$ a quick examination through the following discussions with two $l$'s.
        
\noindent        \textbf{1.} Given that $\Omega$ is extensible. Suppose $\base=4$. Let $u'=111111$ and $v'=1111$ be two words in $\mcXp$. Consider $\alpha = 7,$ a prime number greater than $\base = 4$ and $\xi(u')=6$. It follows from the extensibility of $\Omega$ that for any $k \in \mathbb{N}_0$ there exists $x \in \mcXp$ such that $x|_{s(u')}=u'$ and $(\Pi_{\vert u' \vert \alpha \base^k} x)|_{s(v')}=v'$. 
        
\noindent         \textbf{2.} Suppose $\base=2$ and $\mcXp$ is transitive. To show that $u=101 \in B(\Omega)$ is extensible at position $5$, pick $\alpha = 1, k=2$, $u'=0^{32}$, and $v'=1001$, there exists $y \in \mcXp$ such that $y|_{s(u')}=u'$ and $(\Pi_{\vert u' \vert \alpha \base^k} y)|_{s(v')}=v'$. Let $x:=(y_{\alpha\base^{(k+i)}})_{i \in \mathbb{N}}$. Then $x_{[5,7]}=101$.
\end{example}

The following example provides an intuitive viewpoint for examining Theorem \ref{thm:weak_mixing_equivalence}.

        \begin{figure}[t]
\includegraphics[scale=1]{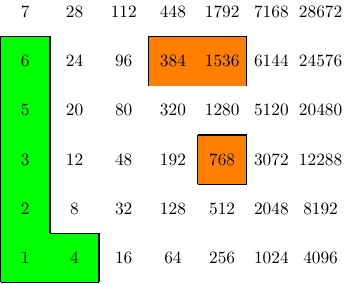}
                \caption{The weakly mixing property of $\Omega$ implies $\mcXp$ is $\base$-directional mixing. When $\alpha \base^k = 1 \cdot 4^3$, the (partial) supports of $u'$ and $x|_{\vert u' \vert \alpha \base^k s(v')}$ are colored in green and orange, respectively. See Example \ref{ex:weak_mixing} for more details.}
                \label{fig:example_weakmixing_converse}
        \end{figure}

\begin{example} \label{ex:weak_mixing}
        For each $P \subset \mathbb{N}_0$ with $0 \in P$, the \emph{spacing shift} introduced in \cite{lau1972weak} is defined as 
        \begin{equation}\label{eqn:spacing_shift}
        \Sigma_P=\{s \in \{0,1\}^\mathbb{N}:s_i=s_j=1 \Rightarrow |j-i| \in P\}.
        \end{equation}
        Let $P=\mathbb{N}_0 \setminus \{2+10^N: N \in \mathbb{N}_0\}$ and $\Omega=\Sigma_P$. Then $\Omega$ is weakly mixing but not mixing since $P$ is thick (see \cite{banks2013dynamics} for more details). We use the following examples to show that \textbf{1.} $\Omega$ is weakly mixing if $\mcXp$ is $\base$-directional mixing; \textbf{2.} $\mcXp$ is $\base$-directional mixing if $\Omega$ is weakly mixing.
        
\noindent         \textbf{1.} Let $\base=2$ and $u_{1}=11, u_{2}=111,v_{1}=111,v_{2}=11 \in B(\Omega)$. It is easily seen that $u'=1110010000010000, v'=111101 \in B(\mcXp)$ and $u'|_{\Lambda_1}=u_{1}$, $u'|_{\Lambda_3}=u_{2}$, $v'|_{\Lambda_1}=v_{1}$, and $v'|_{\Lambda_3}=v_{2}$. More specifically, with $\alpha = 1$ and $k=4$, there is an $x \in \mcXp$ given that $x|_{s(u')}=u'$ and $(\Pi_{\vert u' \vert \alpha \base^k} x)|_{s(v')} = (\Pi_{\base^7} x)|_{s(v')}=v'$ and $x_i=0$ otherwise. Therefore, $u_{1},v_{1}$ are connected in $x|_{\Lambda_1}$ and $u_{2},v_{2}$ are connected in $x|_{\Lambda_3}$.
        
\noindent         \textbf{2.} Let $\base=4$ and $u'=111111, v'=1111 \in B(\mcXp)$. Then $|u'| = \alpha_1 l^{k_1} = 6 \cdot 4^0, \vert v' \vert =4$, and $\alpha_1 \vert v' \vert =24$. It follows from Corollary \ref{cor:p_offset} that for every $\alpha \in \mathbb{N}, k \in \mathbb{N}_0$, and $i \in A_{\base, \vert v' \vert}$, $i \vert u' \vert \alpha \base^k = j \base^{k_1 + k + c}$ for some $j \in A_{\base} \in \mathbb{N}_0$ and $c \leq 2$. To connect $u'|_{\Lambda_i}$ and $v'|_{\Lambda_j}$ for $1 \le i \le \xi(u')$ and $1 \le j \le \xi(v')$, let $k=4$ be assumed so that $[k, k+c+\max_i {\vert u'|_{\Lambda_i} \vert}+\max_i {\vert v'|_{\Lambda_i} \vert}] \subset [4,10] \subset P$. Then for each $\alpha \in A_4$ there exists $x \in \mcXp$ such that $x|_{s(u')}=u'$ and $(\Pi_{4^3 \alpha \vert u' \vert} x)|_{s(v')}=v'$ and $x_i=0$ elsewhere. For the case where $\alpha = 1$, the supports of $u'$ and $(\Pi_{4^3\vert u' \vert} x)|_{s(v')}$ are colored in green and orange respectively in Figure \ref{fig:example_weakmixing_converse}. 
\end{example}

Next, we use the following example to verify Theorem \ref{thm:mixing_equivalence}.

        \begin{figure}[t]
\includegraphics[scale=1]{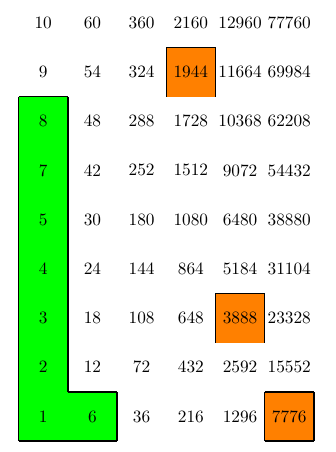}
                \caption{The figure gives an graphical representation of Example \ref{ex:mixing}. Two words $u'=11111111, v'=11111 \in B(\mcX_{\Omega}^{(6)})$ can be concatenated when $\alpha \base^k=243$, where the support of $u'$ and $(\Pi_{\vert u' \vert \alpha \base^k} x)|_{s(v')}$ are colored in green and orange respectively.}
                \label{fig:example_mixing}
        \end{figure}

\begin{example} \label{ex:mixing}
        Let $P=\{0\} \cup \left[3,\infty\right)$ and $\Omega=\Sigma_P$. Then $\Omega$ is mixing since $P$ is cofinite (cf.~\cite{banks2013dynamics}). In this case, given $u,v \in B(\Omega)$, if $x \in \alphabet^\mathbb{N}$ satisfying $x_i=0$ except $x|_{[1,\vert u \vert]}=u$ and $x|_{[\vert u \vert+3+1,\vert u \vert+3+\vert v \vert]}=v$, then $x \in \Omega$. We verify directly that \textbf{1.} $\mcXp$ is mixing if $\Omega$ is mixing and that \textbf{2.} $\Omega$ is mixing if $\mcXp$ is mixing.
        
\noindent         \textbf{1.} Suppose $\base=6$, $u'=11111111$ and $v'=11111$. Then, for every $\alpha \base^k \ge 6^3=216$, there exists $x \in \mcXp$ such that $x|_{s(u')}=u'$ and $(\Pi_{\vert u' \vert \alpha \base^k} x)|_{s(v')}=v'$. For example, when $\alpha \base^k = 243$ the support of $u'$ and $v'$ in $x$ are colored in green and orange respectively in Figure \ref{fig:example_mixing}. It is seen there that $s(u')$ and $s(v')$ are indeed separated at a distance greater than or equal to $3$.
        
\noindent         \textbf{2.} Let $\base=2$. Let $u=1011$ and $v=111$. It is possible to find $u' :=10010001\in B_{l^{|u|-1}}(\mcXp)$ and $v' :=1101\in B_{l^{\vert v \vert-1}}(\mcXp)$ such that $u'|_{\Lambda_1} = u$ and $v'|_{\Lambda_1}=v$. Then, for every $\alpha \base^k \ge \base^3$ there is $y \in \mcXp$ such that $y|_{s(u')}=u'$ and $(\prod_{|u'| \alpha \base^k} y)|_{s(v')}=v'$. In particular, when for every $m \in \mathbb{N}_0$, $\alpha \base^k = \base^{3+m}$ there is $x:=y|_{\Lambda_1} \in \Omega$ so that $x_{[1, \vert u \vert]}=u$ and $x_{[m+1,m+\vert v \vert]}=v$.  
\end{example}

Notably, none of above relations are equivalent to transitivity of $\Omega$. One may refers to Remark \ref{rmk:omega_transitive} to see that it is not equivalent to Theorem \ref{thm:extensibility_equivalence}. As for Theorem \ref{thm:weak_mixing_equivalence}, an example is given as follows.

\begin{example} \label{ex:transitivity}
        Let $\Omega = \mcX_\mathcal{F} \subset \Sigma_{2}^\mathbb{N}$ be defined by forbidden set $\mathcal{F}=\{00,11\}$. Then, $\Omega$ is transitive yet neither totally transitive nor weakly mixing, and $\mcXp$ satisfies all properties in Theorem \ref{thm:extensibility_equivalence}. However, it is not $\base$-directional mixing. It can be verified by considering patterns $u'=0110$, $v'=1011$ in $\mcXpp$, and $\alpha = 1$. If such $y \in \mcXpp$ exists, then $k$ is required to be even for $u'|_{\Lambda_1}$ and $v'|_{\Lambda_1}$, and odd for $u'|_{\Lambda_3}$ and $v'|_{\Lambda_3}$. This contradicts the existence of $k$. Nevertheless, it is consistent with Theorem \ref{thm:weak_mixing_equivalence} and with Theorem \ref{thm:extensibility_equivalence}.
\end{example}

\section{Proofs of Main Theorems}

This section is devoted to demonstrating the main theorems of this paper. We start from the equivalence between extensibility of $\Omega$ and transitivity of $\mcXp$.

        \begin{proof} [Proof of Theorem~\ref{thm:extensibility_equivalence}]
                \begin{figure}[t]
                        \begin{subfigure}[b]{.4\textwidth}
                                \centering
                                \includegraphics[]{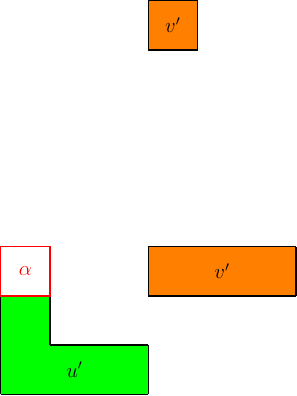}
                                \caption{}
                        \end{subfigure}
                        \begin{subfigure}[b]{.4\textwidth}
                                \centering
                                \includegraphics[]{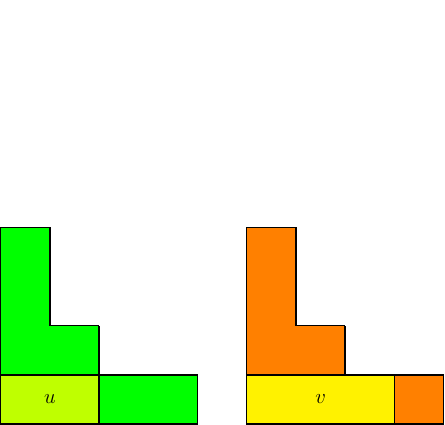}
                                \caption{}
                        \end{subfigure}
                        \caption{An illustration for the equivalence between extensibility of $\Omega$ and transitivity of $\mcXp$. (A) Whenever $\Omega$ is extensible, the transitivity of $\mcXp$ can be derived by separate the given words. (B) If $\mcXp$ is transitive, then $\Omega$ is extensible.}
                        \label{fig:fig3}
                \end{figure}
                
                The theorem is proved in the order (3) $\Rightarrow$ (2) $\Rightarrow$ (1) $\Rightarrow$ (3) and (1) $\Rightarrow$ (5) $\Rightarrow$ (4) $\Rightarrow$ (1). The idea of the proof is referred to Figure \ref{fig:fig3}.
                
                (3) $\Rightarrow$ (2). It follows directly from definition.
                
                (2) $\Rightarrow$ (1). We prove that for $u \in B(\Omega)$ and $m \in \mathbb{N}_0$ there is $x \in \Omega$ such that $x_{[m+1,m+\vert u \vert]}=u$. Let $u',v' \in B(\mcXp)$ such that $\vert u' \vert =\base^m$,  $\vert v' \vert =\base^{\vert u \vert-1}$ and $v'|_{\Lambda_1}=u$. By transitivity of $\mcXp$, there are $\alpha \in A_\base$, $k \in \mathbb{N}_0$ and $y \in \mcXp$ such that $y|_{s(u')} = u'$ and $(\Pi_{\vert u' \vert \alpha \base^k} y)|_{s(v')} = v'$. The proof is completed by letting $x_i := (y_{\alpha \base^{k+i-1}})$.
                
                (1) $\Rightarrow$ (3). We prove that for arbitrary blocks $u', v' \in B(\mcXp)$, and for $k \in \mathbb{N}_0$, there exists $y \in \mcXp$ such that $y|_{s(u')} = u'$ and $(\Pi_{\vert u' \vert \alpha \base^k} y)|_{s(v')} = v'$ whenever $\alpha \in \mathbb{P} \setminus \{\base\}$ with $\alpha > \xi(u')$. Note that for each $\Lambda_i$ with $1 \le i \le \xi(u')$, $(\vert u' \vert \alpha \base^k \Lambda_j) \cap \Lambda_i = \emptyset$ is always the case, so the existence of $y$ is guaranteed by the extensibility of $\Omega$. 

(1) $\Rightarrow$ (5). We claim that for $u', v' \in B(\mcXp)$ and for any $\alpha \in A_{\baseq}$ there exists $y \in \mcXp$ such that $y|_{s(u')} = u', (\Pi_{\vert u' \vert \alpha \baseq^k} y)|_{s(v')} = v'$ whenever $p^k > \xi(u)$, $p|L$ but $p \nmid l$. Since $p^k > \xi(u)$, it follows immediately that
$$
|u'| \alpha L^k \mathbb{N} \cap \Lambda_i = \varnothing \quad \text{for} \quad 1 \leq i \leq \xi(u').
$$
In other words, $\vert u' \vert \alpha L^k s(v')$ is a subset of $\mathbb{N} \setminus \cup_{1 \leq i \leq \xi(u')} \Lambda_i$. Therefore, for each $i \in A_\base$, there exists $x_i \in \Omega$ such that
$$
\left\{
\begin{aligned}
&x_i|_{[1,\left|s(u'|_{\Lambda_i})\right|]} = u'|_{\Lambda_i}, && \text{if } 1 \leq i \leq \xi(u'); \\
&x_i|_{\log_{\base} {\frac{\vert u' \vert \alpha L^k i}{i'}}+[1,\left|v'|_{\Lambda_i}\right|]} = v'|_{\Lambda_i}, && \text{otherwise, where } \vert u' \vert \alpha L^k \Lambda_i \subset \Lambda_{i'};
\end{aligned}
\right.
$$
since $\Omega$ is extensible. Let $y \in \alphabet^\mathbb{N}$ be defined by $y|_{\Lambda_i} = x_i$. Then $y \in \mcXp$ is the desired result.
                
                (5) $\Rightarrow$ (4). It holds automatically since (4) is a particular case of (5).
                
                (4) $\Rightarrow$ (1). We prove that for $u \in B(\Omega)$ and $m \in \mathbb{N}_0$ there is $x \in \Omega$ such that $x_{[m+1,m+\vert u \vert]}=u$. Let $u',v' \in B(\mcXp)$ such that $\vert u' \vert =\base^{m+1}$, that $\vert v' \vert =\base^{\vert u \vert-1}$ and that $v'|_{\Lambda_1}=u$. By $\baseq$-directional mixing property of $\mcXp$, there is $k \in \mathbb{N}_0$ and $y \in \mcXp$ such that $y|_{s(u')} = u'$ and $(\Pi_{\vert u' \vert \baseq^k} y)|_{s(v')} = v'$. Suppose that $\vert u' \vert \baseq^k = i \base^{c + m + 1}$ for some $c \in \mathbb{N}_0$, i.e., $|u'|L^k \in \Lambda_i$ for some $l \nmid i$. The proof is complete by letting $x := (y_{i \base^{c+j}})_{j \in \mathbb{N}}$.
        \end{proof}

        
        \begin{figure}[t]
                \begin{subfigure}[b]{.4\textwidth}
                        \centering
                        \includegraphics[]{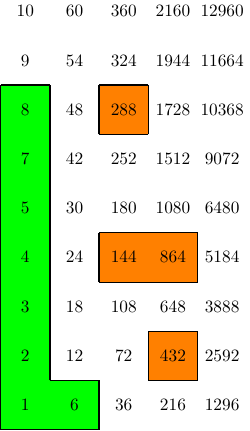} 
                        \label{fig:example_nonleftaligned}
                        \caption{$u \in B_8(\mathsf{X}_{\Omega}^{(6)})$}
                \end{subfigure}
                \begin{subfigure}[b]{.4\textwidth}
                        \centering
                        \includegraphics[]{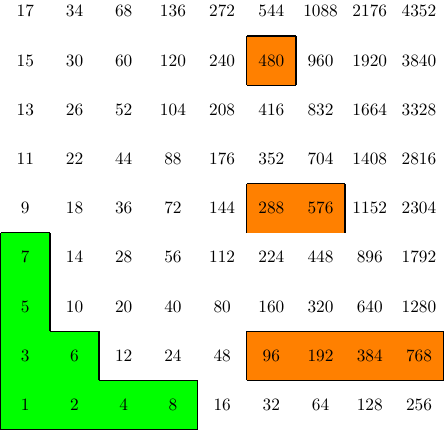} 
                        \label{fig:example_leftaligned}
                        \caption{$u \in B_8(\mathsf{X}_{\Omega}^{(2)})$}
                \end{subfigure}
                \caption{Suppose $x \in \mcXp$ and $u \in B_8 (\mcX_\Omega^{(l)})$. (A) $(x|_{\Lambda_i \cap 144 s(u)})$'s are left-aligned since $6$ is not a prime. (B) $(x|_{\Lambda_i \cap 96 s(u)})$'s are left-aligned since $2$ is a prime.}
                \label{fig:fig4}
        \end{figure}
        
        Next, we link the weakly mixing property of $\Omega$ and the $l$-directional mixing property of $\mcXp$. Proposition \ref{prop:space_partition} implies that the multiplicative transformation breaks the topological structure of pattern even more significantly in the case $\base$ is not a prime than the case $\base$ is a prime. More precisely, for each $i \in A_l = \mathbb{N} \setminus \base \mathbb{N}$, the product $(\vert u' \vert \alpha \base^k) i$ can be represented as $j \base^{n+k+c}$ for some $\base \nmid j$, where extra ``offset'' $c = c_{\alpha}$ is introduced. Hence, $(\vert u' \vert \alpha \base^k) s(v'|_{\Lambda_{i}})$ is not ``left-aligned'' (see Figure \ref{fig:fig4}). Lemma \ref{lem:p_offset} shows that the collection of these offsets $\{c_{\alpha}\}_{\alpha \in A_l}$ is bounded.

        \begin{lemma} \label{lem:p_offset}
                Given any $N \in \mathbb{N}$, there exists $M \in \mathbb{N}$ so that $A_\base A_{\base,N}:=\{a b: a \in A_\base, b \in A_{\base,N}\} \subset \cup_{i=0}^{M} \base^i A_l$.
        \end{lemma}
        \begin{proof}
By the fundamental theorem of arithmetic, there exists the unique prime factorization of $l$ as $\base = p_1^{m_1} p_2^{m_2} \ldots p_r^{m_r}$. Let
$$
M=\max \{\lfloor \frac{k_{p_i}}{m_i} \rfloor: 1 \le i \le m, k \in A_{\base, N}\} + 1,
$$
where $k_{p_i} = \max \{m \in \mathbb{N}_0: p_i^m | k\}$. Given $a \in A_\base$ and $b \in A_{\base,N}$, there exists $1 \leq i \leq r$ such that $\lfloor \frac{a_{p_i}}{m_i} \rfloor = 0$. Since $\lfloor \frac{b_{p_i}}{m_i} \rfloor \le M-1$, it comes immediately that $\lfloor \frac{(a b)_{p_i}}{m_i} \rfloor \le M$. Hence, $a b \in \cup_{i=0}^{M} \base^i A_\base$ and the proof is complete.
        \end{proof}
        
        \begin{corollary} \label{cor:p_offset}
                Suppose $u', v' \in B(\mcXp)$ with $\vert u' \vert = \alpha_1 \base^{k_1}$. There exists $M \in \mathbb{N}$ such that for every $\alpha \in \mathbb{N}, k \in \mathbb{N}_0$, and $i \in A_{\base, \vert v' \vert}$, $i \vert u' \vert \alpha \base^k = j \base^{k_1 + k + c}$ for some $j \in A_{\base}$ and $c \le M$.
        \end{corollary}
        \begin{proof}
Observe that there exists $M_1 \in \mathbb{N}$ such that $\alpha_1 A_{\base, |v'|} \subseteq \cup_{i=0}^{M_1} \base^i A_{\base, \alpha_1 \vert v' \vert}$.
                Lemma \ref{lem:p_offset} shows that there exists $M_2 \in \mathbb{N}$ so that $A_\base A_{\base,\alpha_1 \vert v' \vert} \subset \cup_{i=0}^{M_2} \base^i A_\base$. It follows that
                \begin{equation*} \label{eqn:offset_estimation}
                        \vert u' \vert A_\base \base^k A_{\base, \vert v' \vert} \subseteq \base^{k+k_1} A_\base \cup_{i=0}^{M_1} \base^i A_{\base, \alpha_1 \vert v' \vert} \subseteq \cup_{i=0}^{M_1+M_2} \base^{k+k_1+i} A_{\base},
                \end{equation*}
                the desired result follows by letting $M=M_1+M_2$.
        \end{proof}

With the estimation of the offsets, we are ready for the proof of Theorem \ref{thm:weak_mixing_equivalence}.

        \begin{proof}[Proof of Theorem~\ref{thm:weak_mixing_equivalence}]        
The proof is divided into three parts. First we show that (1) $\Leftrightarrow$ (2). After demonstrating the equivalence of (1), (3), and (4), it follows that (1) $\Leftrightarrow$ (5).

(1) $\Rightarrow$ (2).
Given $u', v' \in B(\mathsf{X}_{\Omega}^{(l)})$ with $|u'| = \alpha_1 l^{k_1}$ for some $\alpha_1 \in A_{\base}$, Corollary \ref{cor:p_offset} indicates there exists $M \in \mathbb{N}$ such that for $\alpha \in A_\base, k \in \mathbb{N}_0$, and $i \in A_{l, |v'|}$,
$$
i |u'| \alpha l^k = i' l^{k_1 + k + c} \quad \text{for some } i' \in A_l \text{ and } c \leq M.
$$
Let
\begin{align*}
        \Delta = \{(u'|_{\Lambda_i}, w^{(r,j)} v'|_{\Lambda_j}): &1 \leq i \leq \xi(u'), 1 \leq j \leq \xi(v'), \\
        &0 \leq r \leq M\},
\end{align*}
be a finite collection of pairs of blocks in $\Omega$, where $\epsilon$ is the empty word and $w^{(r,j)} \in B(\Omega)\cup \{\epsilon\}$ is chosen so that $w^{(r,j)} v'|_{\Lambda_j} \in B(\Omega)$. Since $\mathsf{X}_{\Omega}^{(l)}$ is weakly mixing, there exists $K \in \mathbb{N}$ such that for $(\overline{u}, \overline{v}) \in \Delta$, there exists $\overline{w} \in B_{K-\vert \overline{u} \vert}(\Omega)$ such that $\overline{u} \overline{w} \overline{v} \in B(\Omega)$ by Proposition \ref{prop:weak_mixing}. Note that $K \ge \vert u|_{\Lambda_i} \vert$ for every $i \in A_{\base, \vert u' \vert}$ and so $K \ge k_1$.

Next we show that for each $\alpha \in A_{\base}$ there exists $x \in \mathsf{X}_{\Omega}^{(l)}$ such that $x|_{s(u')} = u'$ and $(\Pi_{|u'| \alpha l^{K-k_1}} x) |_{s(v')} = v'$. Observe that $(\Pi_{|u'| \alpha l^{K-k_1}} x) |_{s(v')} = v'$ if and only if 
$$
x_{j'}|_{\log_\base \frac{\vert u' \vert \alpha \base^{K-k_1} j}{j'} + \left[1, \vert v'|_{\Lambda_j} \vert \right]} = v'|_{\Lambda_j} \quad \text{for} \quad 1 \leq j \leq \xi(v'),
$$
where $j' \in \mathbb{N}$ satisfies $(\vert u' \vert \alpha \base^{K-k_1}) \Lambda_j \subset \Lambda_{j'}$ and $x_{j'} := x|_{\Lambda_{j'}}$. The construction of $x$ is as follows. For $1 \leq i \leq \xi(u')$, if there exists $1 \leq j \leq \xi(v')$ such that $|u'| \alpha l^K \Lambda_j \subseteq \Lambda_i$, the above discussion implies there exists $x_i$ such that
$$
x_i|_{s(u'|_{\Lambda_i})} = u'|_{\Lambda_i} \text{ and } x_i|_{\log_\base \frac{\vert u' \vert \alpha \base^{K-k_1} j}{i} + \left[1, \vert v'|_{\Lambda_j} \vert \right]} = v'|_{\Lambda_j},
$$
which means that $x_i|_{\log_\base \frac{\vert u' \vert \alpha \base^{K-k_1} j}{i} + m} = v'_{j \base^{m-1}}$ for $m \in [1, \vert v'|_{\Lambda_j} \vert ]$; otherwise, the existence of $x_i$ comes from the extensibility of $\Omega$. For the case where $i > \xi(u')$, the existence of $x_i$ also comes from the extensibility of $\Omega$. The desired $x \in \mathsf{X}_{\Omega}^{(l)}$ then follows by letting $x|_{\Lambda_i} = x_i$.
            
            (2) $\Rightarrow$ (1). To show that for $u_{1}, u_{2}, v_{1},v_{2} \in B(\Omega)$ there are $x_{1}, x_{2} \in \Omega$ and $k \in \mathbb{N}$ such that $x_{i}|_{s(u_{i})}=u_{i}$ and that $x_{i}|_{k+s(v_{i})}=v_{i}$ for $i=1,2$, let $u',v' \in B(\mcXp)$ such that $u_1, u_2, v_1, v_2$ are subword of $u'|_{\Lambda_1}, u'|_{\Lambda_{l+1}}, v'|_{\Lambda_1}, v'|_{\Lambda_{l+1}}$ respectively, and that $\vert u' \vert =\base^{k_1}$ for some $k_1 \in \mathbb{N}$. Since $\mcXp$ is $l$-directional mixing, there is a $k \in \mathbb{N}$ and an $y \in \mcXp$ such that $y|_{s(u')} = u'$ and $(\Pi_{\vert u' \vert \base^k} y)|_{s(v')} = v'$. The proof is completed by letting $x_1=y|_{\Lambda_1}$ and $x_2=y|_{\Lambda_{l+1}}$.

The discussion of (1) $\Rightarrow$ (3) and (4) $\Rightarrow$ (1) are similar to that of (1) $\Rightarrow$ (2) and (2) $\Rightarrow$ (1), respectively. Since (4) is a special case of (3), the equivalence of (1), (3), and (4) then follows.
            
The demonstration of (1) $\Leftrightarrow$ (5) is analogous to the derivation of (1) $\Leftrightarrow$ (3) above together with the proof of Theorem \ref{thm:extensibility_equivalence} (1) $\Leftrightarrow$ Theorem \ref{thm:extensibility_equivalence} (5). Thus the detailed elucidation is omitted for the sake of compactness.
        \end{proof}


We finish this section with the proof that $\Omega$ is mixing if and only if $\mcXp$ is mixing.

        \begin{proof} [Proof of Theorem~\ref{thm:mixing_equivalence}]
                The theorem is proved in the order (1) $\Rightarrow$ (2) $\Rightarrow$ (3) $\Rightarrow$ (1).
                        
(1) $\Rightarrow$ (2). Suppose $u', v' \in B(\mcXp)$ are given. Let $u_i = u'|_{\Lambda_i}, v_j = v'|_{\Lambda_j}$ for $i \in A_{\base, \vert u' \vert}, j \in A_{\base, \vert v' \vert}$. Since $\Omega$ is mixing and $A_{\base, \vert u' \vert}$, $A_{\base, \vert v' \vert}$ are finite, there exists $N_0 \in \mathbb{N}$ such that for $m \geq N_0$ there exists $x = x(\overline{u}, \overline{v}) \in \Omega$ such that $x|_{[1, |\overline{u}|]} = \overline{u}, x|_{[m+|\overline{u}|+1, m+|\overline{u}|+|\overline{v}|]} = \overline{v}$, provided $\overline{u} = u_i, \overline{v} = v_j$ for some $i \in A_{\base, \vert u' \vert}, j \in A_{\base, \vert v' \vert}$. We claim that for $\alpha \in \mathbb{N}, k \in \mathbb{N}_0$ such that $\alpha \base^k \ge \base^{N_0}$ there is $y \in \mcXp$ satisfying $y|_{s(u')} = u'$ and $(\Pi_{\vert u' \vert (\alpha \base^k)} y)|_{s(v')}= v'$, which is equivalent to mixing property in $\mcXp$ by Proposition \ref{prop:mixing-transitive-multiplicative-shift}. Similar to the proof of Theorem \ref{thm:weak_mixing_equivalence}, it suffices to show that whenever $(\vert u' \vert \alpha \base^k) \Lambda_j \subset \Lambda_i$ for some $1 \leq i \leq \xi(u')$ and $1 \leq j \leq \xi(v')$, there exists $x_i \in \Omega$ such that
$$
x_i|_{\left[1, \vert u'|_{\Lambda_i} \vert\right]} = u'|_{\Lambda_i} \quad \text{and} \quad x_i|_{\log_\base \frac{\vert u' \vert \alpha \base^k j}{i} + \left[1, \vert v'|_{\Lambda_j} \vert\right]} = v'|_{\Lambda_j}.
$$
Equivalently, we need to show that $\log_\base \frac{L_2}{i} - \log_\base \frac{L_1}{i} \ge N_0$, where $L_1=\max (s(u') \cap {\Lambda_i})$ and $L_2=\min (\vert u' \vert \alpha \base^k (s(v') \cap \Lambda_j))$. Indeed,
$$
\log_\base \frac{L_2}{i} - \log_\base \frac{L_1}{i} = \log_\base \frac{\vert u' \vert}{L_1} \alpha \base^k j \ge m.
$$
Therefore, $\Omega$ being mixing implies that $\mcXp$ is mixing.
                
                (2) $\Rightarrow$ (3). This could be proved by choosing proper $\alpha$ or $k$ in (2).

(3) $\Rightarrow$ (1). Given $u, v \in B(\Omega)$, let $u', v' \in B(\mcXp)$ such that $|u'| = \base^{|u|-1}, |v'| = \base^{|v|-1}$, $u'|_{\Lambda_1} = u$, and $v'|_{\Lambda_1} = v$. Hence, there exists $N \in \mathbb{N}$ such that for $\alpha \in A_\base, k \geq N$ there exists $y \in \mcXp$ such that $y|_{[1, |u'|]} = u'$ and $(\Pi_{|u'| \alpha \base^k} y)|_{[1, |v'|]} = v'$. We claim that for $m \geq  N$ there exists $x \in \Omega$ such that $x|_{[1, |u|]} = u, x|_{[|u|+m+1, |u|+m+|v|]} = v$. Indeed, let $\alpha = 1$ and $k = m+1 > N$, there exists $y \in \mcXp$ such that $y|_{s(u')} = u'$ and $(\Pi_{|u'| \alpha \base^k} y)|_{s(v')} = v'$. In other words, $(y|_{\Lambda_1})|_{[1, |u|]} = u$ and
$$
(\Pi_{|u'| \alpha \base^k} y)_i = y_{i \base^{|u| + m}} = v_j \quad \text{for} \quad i= \base^{j-1}, j = 1, \ldots, |v|.
$$
The proof is then complete by letting $x = y|_{\Lambda_1}$.
        \end{proof}

\section{Summary and Discussion}

\begin{table}[t] 
        \begin{tabular}{c|ccc}
                \diagbox[width=10em]{$\Omega$}{$\mcXp$} & transitivity & $\base$-directional mixing & mixing \\
                \hline
                extensibility & EQ & T & T  \\
                transitivity & F & T & T \\
                weakly mixing & F & EQ & T \\
                mixing & F & F & EQ \\
        \end{tabular}
        \caption{Summary of the main results. In this table, `T' means that the property in $\Omega$ implies the property in $\mcXp$ and `F' means the opposite, and `EQ' means two properties are equivalent.}
        \label{table:summary}
\end{table}

Suppose $\Omega$ is a traditional shift space and $\mcXp$ is the corresponding multiplicative shift space for some $l > 1$. We investigate the relations between the mixing properties of $\Omega$ and $\mcXp$. After introducing the $l$-directional mixing property, we reveal some if-and-only-if connection between mixing properties of two systems. Table \ref{table:summary} summarizes the main results of this paper. It is seen that there are still open problems remained to be studied. We list these problems of interest in the following, some of which are in preparation.

\begin{question}
        In Theorem \ref{thm:weak_mixing_equivalence}, (5) is equivalent to the others if $\base$ is a prime number. Does this hold for arbitrary $\base > 1$?
\end{question}

\begin{question}
        Is there any equivalent condition for $\mcXp$ as transitivity of $\Omega$?
\end{question}

\begin{question}
        Do Theorems \ref{thm:extensibility_equivalence}, \ref{thm:weak_mixing_equivalence}, and \ref{thm:mixing_equivalence} hold for two-sided multiplicative subshift $\mcXp \subset \alphabet^{\Zint}$?
\end{question}


\bibliographystyle{amsplain}
\bibliography{MultiplicativeShiftMixing}

\end{document}